\documentclass[12pt,draftcls,onecolumn]{IEEEtran}

\ifCLASSINFOpdf
\else
\fi

\usepackage{algorithmic}
\usepackage{algorithm}
\usepackage[cmex10]{amsmath}
\usepackage{amssymb,amsmath}
\usepackage{psfrag,graphics,epsfig,epic,color,subfigure}
\usepackage{graphicx} 
\usepackage{accents}

\newtheorem{Theorem}{Theorem}
\newtheorem{Corollary}{Corollary}
\newtheorem{Lemma}{Lemma}
\newtheorem{Proposition}{Proposition}
\newtheorem{Remark}{Remark}
\newenvironment{proof}{{\em Proof:} \ \ }{\begin{flushright}$\Box$\end{flushright}}

\newcommand{\R}{\mathbb{R}}

\newcommand{\AM}{\mathbf{A}}

\newcommand{\CM}{\mathbf{C}}

\newcommand{\IM}{\mathbf{I}}
\newcommand{\HM}{\mathbf{H}}
\newcommand{\KM}{\mathbf{K}}
\newcommand{\LM}{\mathbf{L}}
\newcommand{\MM}{\mathbf{M}}

\newcommand{\QM}{\mathbf{Q}}

\newcommand{\VM}{\mathbf{V}}

\newcommand{\ZM}{\mathbf{Z}}
\newcommand{\SMM}{\mathbf{\Sigma}}

\newcommand{\av}{\mathbf{a}}
\newcommand{\bv}{\mathbf{b}}

\newcommand{\hv}{\mathbf{h}}

\newcommand{\xv}{\mathbf{x}}
\newcommand{\yv}{\mathbf{y}}
\newcommand{\vv}{\mathbf{v}}
\newcommand{\wv}{\mathbf{w}}
\newcommand{\zv}{\mathbf{z}}

\begin{document}

\title{Worst-case Prediction Performance Analysis of the Kalman Filter}

\author{Sholeh~Yasini
        and~Kristiaan~Pelckmans
\thanks{This work is supported by Swedish Research Council under contract
	621-2007-6364.
	S. Yasini (corresponding author) and K. Pelckmans are with the 
Division of Systems and Control, Department of Information Technology,
Uppsala University, Box 337, SE-751 05 Uppsala, Sweden. (Emails: sholeh.yasini@it.uu.se
, kp@it.uu.se).}}

\maketitle

\begin{abstract}
In this paper, we study the {prediction} performance of the Kalman filter (KF) in a worst-case, {\em minimax} setting { as studied in online machine learning, information - and game theory.} { The aim is to predict the sequence of observations almost as well as the best reference predictor (comparator) sequence in a comparison class. We prove worst-case bounds on the cumulative squared prediction errors using a priori knowledge about the complexity of reference predictor sequence. In fact, the performance of the KF is derived as a function of the performance of the best reference predictor and the total amount of drift occurs in the schedule of the best comparator.}
\end{abstract}


\IEEEpeerreviewmaketitle

\section{Introduction}
\IEEEPARstart{S}{ince} its inception in the 1960s, the Kalman Filter (KF) 
has been one of the most powerful tools in signal estimation and control engineering.
It is a technique for estimating  the unknown states of dynamical 
systems from a set of the resulting noisy measurements, see e.g. \cite{Kailath} for an overview. 
{ It is well-known that the KF yields the Minimum Variance Estimator (MVE) if the noise terms are assumed to be zero-mean white Gaussian processes. However, such Gaussian assumption may limit the utility of the estimators in situations where the noise terms behave quite differently. If the Gaussian assumption is discarded, the filter gives the Linear Minimum Variance Unbiased Estimator (LMVUE). 
	The problem of state-estimation in a stochastic setting with unknown distribution on the random terms has been studied (e.g., see \cite{Aliev}, \cite{Maryak95}, \cite{Maryak}, \cite{Spall95} and references therein). 
	The approach presented in \cite{Aliev}, \cite{Maryak95} is based on asymptotic distribution theory for state estimator in the absence of the usual Gaussian assumptions on the noise terms. 
	\cite{Maryak} and \cite{Spall95} show how inequalities such as Chebyshev inequality from probability theory can be employed for characterizing the uncertainty bounds of the estimation error in a general distribution free setting. 
	However, these methods assume an underlying stochastic dynamical model with known finite second order moments for the noise terms. Moreover, they assume that the random terms (the initial state, process noise and measurement noise) are mutually independent. In many applications, however, one is faced with lack of statistical knowledge of noise terms. A natural question to ask is what the performance of the KF will be with respect to disturbance variation and lack of statistical knowledge of the disturbance terms.} Some $H_{\infty}$ estimation methods \cite{Hassibi-kf}, \cite{Hassibi-lms} have been developed to study the performance of adaptive filtering without requiring statistical assumptions. Such $H_{\infty}$ estimators are safeguard against the {\em worst-case} disturbance that maximizes the energy gain to the estimation errors.  {These techniques assume the existence of a linear model which together with additive noise generates the outputs. 
	Performance analysis of such approach does not address the situation where the generative models are incorrectly specified.} 

{ In this note, we aim to study the performance of the KF through analyzing its game-theoretic-based setting.
	This work is inspired by the worst-case, {\em minimax} setting in online machine learning as surveyed in \cite{Nikolo}. This approach is based on the {\em universal prediction} perspective which attempts to perform well on any sequence of observation without assuming a generative model of the data. 
	Our approach is essentially different from stochastic state-space analysis in that we abandon the basic assumption that the output sequence is generated by an underlying stochastic process. 
	Instead of making assumption on the data generating process, we ask: can we predict the sequence of outputs online almost as well as the best reference predictor (comparator) {\em in hindsight} (in this case, in hindsight means that the reference predictor corresponds to the output of an offline algorithm with access to all the data simultaneously)?
	
	
	The online algorithm is formulated as an iterated game between the {\em player} and the {\em environment}}. { At each round or {\em trial} of this game, the player
	is challenged to guess an unknown output vector generated by the environment. 
	The player computes its prediction for the output by combining the entries of a known input matrix with the information collected in the past trials. 
	After this, the actual output is revealed, and the player incurs a loss computed according to a fixed {\em convex loss function}, measuring the discrepancy between the player's prediction and the observed output. 
	As no assumption is made on how the environment generates the sequence of outputs, the player would accumulate an arbitrary high loss. 
	To set a reasonable goal, a {\em competitive} approach is adopted: the performance of the player is measured against the performance of the comparator from some comparison class $\mathcal{X}$. 
	In the extreme case, the sequence of observations could be completely random, in which it could be predicted by neither the player nor any reference predictor from the comparison class $\mathcal{X}$. On the other hand, the sequence of outputs might be completely predictable by a reference predictor from $\mathcal{X}$, in which case the player should incur only a small loss before learning to follow that predictor. In general, the goal of the player is to achieve an average loss of prediction that is not too large compared to the average loss of the best offline reference predictor.}

Recently, there has been an increasing interest 
in the machine learning community for analyzing the Least Mean Square (LMS) algorithm \cite{Nikolo-96}, \cite{Kivinen-97} and 
the Recursive Least Square (RLS) algorithm \cite{Froster-99}, \cite{Vovk-01} 
in a worst-case setting, see also \cite{anava2013}. 
{The aforementioned work analyzes the performance of the algorithm in the stationary setting and guarantee that the algorithm performs almost as well as the best offline single reference predictor.}  However, in many real-world applications,
particularly, in an adaptive context not accounting for the effect of parameter 
variations is insufficient. This motivates the study
of the performance of algorithms that are able to compete with 
the best reference predictor that drifts over time. 
Naturally, tracking worst-case bounds are, { in general}, much harder to prove than worst-case bounds for stationary targets. {The tracking problem in the minimax setting has been considered previously: \cite{Herbster-01} derived tracking bounds for general gradient descent algorithms and proposed a generic method for converting the stationary regret bounds into tracking regret bounds.} For online regression algorithms, 
tracking regret bounds were studied and analyzed in
\cite{Moroshko}, \cite{Vaits-11}, \cite{Vaits-15}. \cite{Vaits-11} uses a 
projection step to control the eigenvalues of a
covariance-like matrix using scheduled resets, whereas the
algorithm designed in \cite{Moroshko} is based on a last-step technique 
\cite{Froster-99} for which the eigenvalues 
are controlled implicitly. { However, they lack a mechanism for incorporating dynamical models for effective prediction and tracking performance.}


{ We provide new tracking worst-case bounds for the KF by 
	analysing its game-theoretic-based model. 
}
The proposed tracking bounds
rely on a general notion of the drift or the comparator complexity as it is measured in terms of how much it deviates from a known fixed dynamical model.
{ It is noteworthy that incorporating a predetermined dynamical model in the comparator's complexity, makes the problem more difficult and additional consideration is required for deriving tracking bounds for the KF compared to online regression algorithm \cite{Vaits-11, Vaits-15}.}

The precise setup is given in the next subsection. Subsection B relates the current approach to an
$H_{\infty}$ setting. In Section II 
the main results are given and 
technical comparison with the $H_{\infty}$ setting is given in Section III. 
Section IV illustrates the theoretical results using simulation. Finally, 
the paper is concluded with a discussion of the results in Section V.



\subsection{Problem Formulation}
{ In this subsection, we study the game-theoretic-based formulation of the KF where an adversary (environment) generates the sequence of observations. To establish the worst-case bounds, we combine the state update mechanism from the classical KF with the online learning setting from universal prediction.
	The algorithm is defined by a 4-tuple $(\AM, \CM, \VM, \QM)$ where $\VM \in \R^{p \times p}$, $\QM \in \R^{n \times n}$ are user-defined positive definite matrices, and $\AM \in \R^{n \times n}$, $\CM\in \R^{p \times n}$ are time-invariant matrices assumed to be known and given to the algorithm.}

The game (as implemented here by the KF) proceeds in trials $t = 0, 1, \dots, T - 1$, where $T$ is any arbitrary value.
{ The player maintains a parameter vector (hypothesis), denoted by $\hat \xv_t \in \R^n$, and a positive definite matrix $\SMM_t \in \R^{n \times n}$. In each trial $t$  the player makes a prediction
	(here we concentrate on {\em linear predictors})
	$\hat \yv_t = \CM \hat \xv_t$.
	
	Then, the environment reveals the actual output and the player incurs the corresponding loss $\| \yv_t - \CM \hat \xv_t \|_{\VM^{-1}}^2$, 
	Finally, the player updates 
	its prediction rules as
	\begin{align}
	\hat \xv_{t+1}  = & \AM \hat \xv_t +\underbrace{\AM \left(\SMM_t^{-1}  + 
		\CM^\top \VM^{-1} \CM \right)^{-1} \CM^\top \VM^{-1}}_{\KM_t}   \left(\yv_t - \hat \yv_t \right) {\label{kt}}\\
	\SMM_{t+1}  = & \AM \left(\SMM_t^{-1} + \CM^\top \VM^{-1}\CM \right)^{-1} \AM^\top + \QM, {\label{riccati}}
	\end{align}
	and proceeds to the next round. The update of the parameter vector $\hat \xv_t$ is additive, with the error $(\yv_t - \hat \yv_t)$ scales by the (Kalman) gain $\KM_t$ and the matrix $\SMM_t$ is updated according to the Riccati recursion.} 
The algorithm is initialised by a fixed user-defined $\hat \xv_0$ and $\SMM_0$. Algorithm \ref{algkf} details the precise protocol.

The total loss of the player on the sequence of 
observations $S  = \{\yv_0, \yv_1, \dots\, \yv_{T-1}\}$ is
\begin{equation} 
L_T = \sum_{t = 0}^{T-1} \| \yv_t - \CM \hat \xv_t \|_{\VM^{-1}}^2.
\end{equation}\label{cumloss}

{According to the methodology of worst-case bounds as set out in online learning theory (\cite{Nikolo}), this $L_T$ is compared to the total loss of any reference predictor $\bar p \in \mathcal{X}$ where $\mathcal{X}$ is the comparison class of predictors. 
	To any arbitrary sequence of comparator $\{\bar \xv_0, \bar \xv_1, \dots, \bar \xv_{T} \}$, we associate a linear predictor, defined as $\bar p(\{\bar \xv_t\}) = \CM \bar \xv_t$ with $\bar \xv_t \in \R^n$. Then any set $\mathfrak{X} \subseteq \R^n$ of vectors defines a comparison class $\mathcal{X}$ of linear predictors by $\mathcal{X} = \{\bar p (\bar \xv_t) | \bar \xv_t \in \mathfrak{X}, t = 0, 1, \dots, T\}$. Given a sequence of observations $S  = \{\yv_0, \yv_1, \dots, \yv_{T-1} \}$, the total loss of the reference predictor $\bar p (\{\bar \xv_t\})$ is defined as }

\begin{equation}
V_T = \sum_{t=0}^{T-1} \| \yv_t - \CM \bar \xv_t \|_{\VM^{-1}}^2.
\end{equation}
{Now, the aim of the player is to incur small loss relative to the reference predictor sequence in $\mathcal{X}$.
	Hence, for any $T$, the goal is to obtain an upper bound of the form
	\begin{equation} \label{worstbound}
	L_T \leq \inf_{\{\bar\xv_0, \bar \xv_1 \dots, \bar \xv_{T}  \} \in \mathcal{X}} \{c_1 V_T + c_2 \underbrace{\textrm{ size}_{\AM} (\{ \bar\xv_0, \bar\xv_1, \dots, \bar \xv_{T} \})}_{\text{comparator complexity}} + c_3 \},
	\end{equation}}
where $c_1$, $c_2$, and $c_3$ are positive constants,
and '$\textrm{size}_{\AM} (\{ \bar\xv_0, \bar\xv_1, \dots, \bar \xv_{T} \})$' measures, intuitively, the complexity of the reference predictor in terms of how much it deviates from a {\em given dynamical model} $\AM$:
\begin{equation} \label{comparatorcomplexity}
\textrm{ size}_{\AM} (\{ \bar\xv_0, \bar\xv_1, \dots, \bar\xv_{T} \}) \equiv \sum_{t=0}^{T-1} \| \bar\xv_{t+1} - \AM \bar\xv_t \|^2 ,
\end{equation} 
also referred to as the total drift. {Note that the infimum (minimum) in (\ref{worstbound}) depends on all the observations and can hence correspond to an offline solution computed after observing all the data. }
\begin{Remark}
	{An important feature of the Algorithm 1 is that a fixed and known dynamical model is incorporated in the learning process}. The worst-case bound (\ref{worstbound}) scales proportionally 
	to the comparator complexity (\ref{comparatorcomplexity}),  from a sequence evolving with known dynamics. If the reference predictor follows the dynamics well as described by the system matrix $\AM$, then the $\textrm{size}_{\AM} (\{ \bar\xv_0, \bar\xv_1, \dots, \bar \xv_{T} \})$ would be small and one would get a small bound. If, however, $\textrm{size}_{\AM} (\{ \bar\xv_0, \bar\xv_1, \dots, \bar \xv_{T} \})$ is large (e.g. because of misspecification), the resulting bound will become rather large. This insight is substantiated both theoretically (Section II) as well as experimentally (see Section IV).
\end{Remark}

{ \begin{Remark}
		Whereas our algorithm always predict according to a linear function as hypothesis using a fixed and known matrix $\CM$, the resulting worst-case bound does not require the outputs $\{\yv_t \}$ to come from a linear system. It is worth mentioning that our bounds will be parametrized by the weights of the reference predictor, (i.e., $\{\bar \xv_t \}$). However, the algorithm for which we prove these bounds does not need the vectors $\{\bar \xv_t \}$ as input parameter. 
	\end{Remark}}
	
	{\begin{Remark} \label{Rem2}
			In cases where the constant $c_1 = 1$, one can rewrite (\ref{worstbound}) in terms of the \emph{relative regret} which is defined as the discrepency between the total loss of the player 
			and that of the reference predictor $\bar p (\{\bar \xv_t\})\in \mathcal{X}$
			\begin{equation} \label{eq:reg}
			\mathcal{R}_T (\{\bar \xv_t \}) = L_T - V_T. 
			\end{equation}
			The goal of the player is to minimize the supremum of the regret with respect to the comparison class $\mathcal{X}$ 
			and for an arbitrary sequence of observations $S$ 
			chosen by the environment
			\begin{equation}
			\begin{split}
			\mathcal{R}_T \equiv \sup_{ \{\bar\xv_0, \bar \xv_1, \dots, \bar \xv_{T} \} \in \mathcal{X}}  \mathcal{R}_T(\{\bar\xv_t\}) & = L_T - \inf_{ \{\bar\xv_0, \bar \xv_1 \dots, \bar \xv_{T}  \} \in \mathcal{X}} V_T \\
			& \leq c_2 \textrm{ size}_{\AM} (\{ \bar\xv_0, \bar\xv_1, \dots, \bar \xv_{T} \}) + c_3 ,
			\end{split}
			\end{equation}
			where $c_2$ and $c_3$ are positive constants. We restate the player's aim as having a {\em low regret}, by which we mean that $\mathcal{R}_T$ grows sublinearly with the number of rounds $T$, i.e., $\mathcal{R}_T = o(T)$.
			Intuitively, the regret measures how sorry the player is, {\em after} seeing all the data (i.e., in hindsight), for not having followed the predictions of the best reference predictor in the comparison class $\mathcal{X}$.

		\end{Remark}}
		
		\begin{algorithm} []
			\caption{Game-theoretic-based Kalman filter}
			\label{algkf}
			\begin{algorithmic}
				\REQUIRE $\CM \in \R^{p \times n}$, $\AM \in \R^{n \times n}$. 
				Set $\SMM_0 = \IM_n$ (identity matrix), $\hat \xv_0 = \mathbf{0}$, $\QM \succ 0$, $\VM \succ 0$.   \\
				\FOR{t = 0, 1, \dots}  
				\STATE (1) Player predicts $\hat \yv_t = \CM \hat \xv_t$ 
				\STATE (2) Environment reveals $\yv_t$
				\STATE (3) Player incurs loss $ \lVert \yv_t - \CM \hat \xv_t \rVert_{\VM^{-1}}^2$
				\STATE (4) Player updates estimate as 
				\begin{equation}
				\begin{split}
				\hat \xv_{t+1} & = \AM \hat \xv_t + \AM \left(\SMM_t^{-1}  + 
				\CM^\top \VM^{-1} \CM \right)^{-1} \CM^\top \VM^{-1} \left(\yv_t - \hat \yv_t \right) \\
				\SMM_{t+1} & = \AM \left(\SMM_t^{-1} + \CM^\top \VM^{-1}\CM \right)^{-1} \AM^\top + \QM. \nonumber
				\end{split}
				\end{equation}
				\ENDFOR
			\end{algorithmic}
		\end{algorithm}

		\subsection{Relation to the State-space $H_\infty$ Estimation}
		The worst-case/{\em minimax} 
		analysis in the regret framework is closely related to the $H_{\infty}$ estimation theory, 
		see e.g. \cite{Hassibi-krein, Hassibi-lms, Shaked} and references therein:
		assume that there exists a 
		linear time-invariant (LTI) model explaining the data:
		\begin{align} \label{sysdyn}
		\bar\xv_{t+1} = \AM \bar\xv_{t} + \wv_{t}  \nonumber \\ 
		\yv_t = \CM \bar\xv_t + \vv_t,
		\end{align}
		where $\bar\xv_t \in \R^n$ is an unknown true parameter state vector, $\yv_t \in \R^p$ is the measured output, 
		$\wv_{t} \in \R^n$ is the process noise and $\vv_t \in \R^p$ is
		referred to as observation residual. Let $\bar \zv_t = \LM \bar \xv_t$ be any linear combination
		of the state $\bar \xv_t$ where $\LM \in \R^{q \times n}$ and let $\hat \zv_t = \mathcal F_p (\yv_0, \dots, \yv_{t-1})$ denote an estimate of $\bar \zv_t$ given past observations $\{\yv_0, \dots, \yv_{t-1}\}$. Let $\mathcal{T}_{p,T} (\mathcal{F}_p)$ denote the transfer operator that maps the unknown normalized disturbances 
		$ \{ \SMM_0^{-1/2} (\bar \xv_0 - \hat \xv_0), \{\QM^{-1/2} \wv_t, \VM^{-1/2} \vv_t \}_{t=0}^{T - 1} \} $ to the 
		normalized predicted estimation error $\{\ZM^{-1/2} (\bar \zv_t - \hat \zv_t) \}_{t = 0}^{T - 1}$ where $\SMM_0$, $\QM$, and $\ZM$ are user-defined positive definite weighting matrices and $\hat \xv_0$ denotes an initial guess of $\bar \xv_0$.
		The $H_{\infty}$ norm of $\mathcal{T}_{p,T} (\mathcal{F}_p)$ denoted as $\| \mathcal{T}_{p,T} (\mathcal{F}_p) \|_{\infty}$ is defined as the maximum energy gain from unknown disturbances to the predicted estimation error. An $H_{\infty}$ optimal strategy $\hat \zv_t = \mathcal{F}_p (\yv_0, \dots, \yv_{t-1})$ minimizes 
		$\| \mathcal{T}_{p,T} (\mathcal{F}_p) \|_{\infty}$, and achieves the optimal attenuation factor $\gamma^*$
		\begin{equation}
		\begin{split}
		\gamma_*^2 =& \inf_{\mathcal{F}_p} \| \mathcal{T}_{p,T} (\mathcal{F}_p) \|^2_{\infty} \\
		=& \inf_{\mathcal{F}_p} \sup_{\bar \xv_0, \wv , \vv \in l_2} 
		 \frac{\sum_{t = 0}^{T - 1}\| \bar \zv_t - \hat \zv_t\|_{\ZM}^2}{  \| \bar \xv_0 \|^2_{\SMM_0^{-1}} + \sum_{t=0}^{T - 1}\| \vv_t \|_{\VM^{-1}}^2 + \sum_{t=0}^{T - 1} \| \wv_t \|^2_{\QM^{-1}} } ,
		\end{split}
		\end{equation} 
		where $l_2$ denotes the space of square-summable sequences, and the infimum is taken over all strictly causal estimators $\mathcal{F}_p$. Note that in the above, and as done throughout the paper, we assumed $\hat \xv_0 = \mathbf{0}$. A simpler problem would be to relax the minimization problem, and obtain a sub-optimal solution, i.e., given scalar $\gamma > 0$, find estimation strategies $\hat \zv_t = \mathcal{F}_p (\yv_0, \dots, \yv_{t-1})$ that achieve
		\begin{equation}
		\| \mathcal{T}_{p, T} (\mathcal{F}_p) \|_{\infty} \leq \gamma ,
		\end{equation} 
		for all $T$.
		Results from the $H_{\infty}$ setting state
		that the $H_{\infty}$ filter ensures the desired attenuation level, $\gamma$, provided that the solution of an associated Riccati equation satisfies a suitable feasibility 
		assumption \cite{Hassibi-krein, Shaked}. 
		It is also shown in \cite{Hassibi-kf}
		that the $H_{\infty}$ norm of the RLS algorithm depends on the input-output data while for the LMS and normalized LMS algorithms the $H_{\infty}$ norm is simply unity \cite{Hassibi-lms}. The results further show that for the filtered normalized estimation error $\{\VM^{-1/2} (\CM \bar \xv_t - \CM \hat \xv_{t|t})\}_t$, where $\hat \xv_{t|t}$ denotes the estimate of $\bar \xv_t$, given $\{\yv_0, \dots, \yv_t \}$, the $H_{\infty}$ norm of RLS and KF is bounded by two. Nevertheless, for the predicted estimation error $\{\VM^{-1/2} (\CM \bar \xv_t - \CM \hat \xv_{t})\}_t$, which does not have access to current observations, the $H_{\infty}$ performance of the RLS and KF can be much larger \cite{Hassibi-kf}.
		
		Motivated by the similarity between the worst-case (regret) analysis 
		in online machine learning and $H_{\infty}$ setting in adaptive filtering, Kivinen \emph{et al.} \cite{Kivinen-06} 
		employed techniques of machine learning  
		\cite{Nikolo-96} and applied them in filtering problem 
		to analyze the LMS algorithm by resorting to the study of Bregman divergences. 
		
		This connection formed also the basis for \cite{Hassibi-kf}, 
		which is related to the present study.
		To ensure a fair comparison, a worst-case bound
		is derived based on the $H_{\infty}$ bound of the KF in \cite{Hassibi-kf} and is compared 
		with the proposed bounds in Section III.
		

		\section{Analysis}
		{In this section, we analyse the game-theoretic-based setting of the KF.}
		
		In many cases, $\SMM_t$ (\ref{riccati}) (and hence the Kalman gain $\KM_t$ (\ref{kt})) converges to the steady-state value. The limiting solution $\SMM$ will satisfy the following Discrete Algebraic Riccati Equation (DARE)
		\begin{equation} \label{ARE}
		\SMM = \AM (\SMM^{-1} + 
		\CM^\top \VM^{-1} \CM)^{-1} \AM^\top + \QM,
		\end{equation}
		and 
		\begin{equation} \label{Kbar}
		\KM = \AM \left(\SMM^{-1} + \CM^\top \VM^{-1} \CM \right)^{-1} \CM^\top \VM^{-1} ,
		\end{equation} 
		is the steady-state Kalman gain. We shall be particularly interested in real, symmetric, positive semi-definite solutions of the DARE which gives a stable steady-state filter. 
		
		\begin{Lemma} \label{lem1} \cite{Kailath}
			Provided that $(\CM, \AM)$ is detectable and $(\AM, \QM)$ 
			is stabilizable 
			and $\SMM_0 \succeq 0$, then 
			$\lim_{t \rightarrow \infty} \SMM_t = \SMM$, 
			$\lim_{t \rightarrow \infty} \KM_t = \bar \KM$ exponentially fast
			where $\{\SMM_t\}_t$ is the solution of the Riccati recursion (\ref{riccati}) 
			and $\SMM$ its the unique stabilizing solution of DARE (\ref{ARE}).
		\end{Lemma}
		
		
		\begin{Lemma} \label{lem2}
			Denote the instantaneous drift as
			\begin{equation}{\label{wt}}
			\wv_t =  \bar \xv_{t+1} - \AM \bar \xv_t.
			\end{equation}
			For the steady-state KF, the state estimation error 
			$\tilde\xv_{t+1} = \bar\xv_{t+1} - \hat \xv_{t+1}$ propagates according to the linear system
			\begin{equation} \label{xtilde}
			\tilde \xv_{t+1} = \HM \tilde \xv_{t} + \bar \wv_t, 
			\end{equation}
			with $\HM= \AM - \bar \KM \CM$ the closed-loop system, driven by the process $\bar \wv_t = \wv_t - \bar \KM \vv_t$. 
			If $(\CM, \AM)$ is detectable and $(\AM, \QM)$ is controllable, this closed-loop system is stable.
		\end{Lemma}

		\begin{Lemma} \label{lem3}  
			Suppose $\left(\CM, \AM\right)$ is detectable and $\left(\AM, \QM\right)$ is controllable. Let the KF algorithm \ref{algkf} be run on the sequence $ S = \{\yv_0,  \yv_1 , \dots\, \yv_{T-1}\}$ with the initial values $\hat \xv_0 = \mathbf{0}$ and $\SMM_0 = \IM_n$, 
			and generating the estimates $\{\hat \xv_0, \hat \xv_1, \dots, \hat \xv_T \}$. 
			Then for any $\alpha>0$ and for all sequences of targets $\{\bar\xv_0, \bar \xv_1, \dots, \bar \xv_T \}$, one has 
			\begin{equation} \label{reg}
			\begin{split}
			L_T  \leq \bar r V_T + \bar r \| \bar \xv_0 \|^2 
			 + \bar r a \left( \frac{1}{\alpha} 
			\sum_{t = 0}^{T-1} \| \tilde \xv_t \|^2 + \alpha  \sum_{t=0}^{T-1} \| \bar \xv_{t+1} - \AM \bar \xv_t \|^2\right),
			\end{split}
			\end{equation}
			where 
			\begin{equation} \label{rbar}
			\bar r = \bar \sigma \left( \VM^{-\frac{1}{2}} (\VM + \CM \SMM \CM^\top) \VM^{-\frac{1}{2}}\right), \quad
			a = \bar \sigma (\SMM^{-1}) ,
			\end{equation}
			with $\SMM$ the steady-state value of $\SMM_t$,
			and where $\bar \sigma (.)$ denotes the largest singular value of its argument.
		\end{Lemma}
		\begin{proof}
			We start by forming $\| \tilde \xv_{t+1}\|_{\SMM_{t+1}^{-1}}^2$ as follows
			\begin{equation}
			\begin{split}
			\|\tilde \xv_{t+1} \|_{\SMM_{t+1}^{-1}}^2  & =
			\| \AM \tilde \xv_t  - \AM (\SMM_t^{-1} + \CM^\top \VM^{-1} \CM )^{-1} \CM^\top \VM^{-1} \\
			& . (\yv_t - \CM \hat\xv_t) + \bar \xv_{t+1} - \AM \bar \xv_t \|^2_ {\SMM_{t+1}^{-1}} .
			\end{split}
			\end{equation}
			Using an
			upper bound from Hassibi and Kailath \cite{Hassibi-kf}  (Lemma \ref{lem2a} in Appendix),
			for all $\alpha > 0$ we have that
			\begin{equation}
			\begin{split}
			\|\tilde \xv_{t+1} \|_{\SMM_{t+1}^{-1}}^2&  \leq 
			(1 + \frac{1}{\alpha})  \| \AM \tilde \xv_t  - \AM (\SMM_t^{-1} + \CM^\top \VM^{-1} \CM )^{-1} \CM^\top \VM^{-1} \\
			& . (\yv_t - \CM \hat\xv_t) \|^2_{\SMM_{t+1}^{-1}} + (1 + \alpha) \| \bar \xv_{t+1} - \AM \bar \xv_t \|^2_ {\SMM_{t+1}^{-1}} .
			\end{split}
			\end{equation}
			Now, using lemma \ref{lem1a} in Appendix we substitute $\SMM_{t+1}^{-1}$ by $\bar {\mathbf{\Pi}}_t^{-1}$ in the first term to get
			\begin{equation}
			\begin{split}
			\|  \tilde\xv_{t+1} \|_{\SMM_{t+1}^{-1}}^2  \leq 
			& (1+ \frac{1}{\alpha}) \| \tilde \xv_t \|^2_{\SMM_t^{-1}} \\
			& + (1 + \frac{1}{\alpha}) \tilde\xv_t ^\top \CM^{\top}\VM^{-1} \CM \tilde \xv_t \\
			& + (1 + \frac{1}{\alpha}) \left(\yv_t - \CM \hat \xv_t\right)^\top \VM^{-1} \CM  \\
			& . \left(\SMM_t^{-1} + \CM^\top\VM^{-1} \CM \right)^{-1} \CM^\top\VM^{-1} \left(\yv_t - \CM \hat \xv_t\right) \\
			& - 2 (1+ \frac{1}{\alpha}) \tilde\xv_t^\top \CM^\top \VM^{-1} \left(\yv_t - \CM \hat \xv_t\right) \\
			& + (1 + \alpha) \|  \bar \xv_{t+1} - \AM \bar \xv_t \|^2_{\SMM_{t+1}^{-1}}. \\
			\end{split}
			\end{equation}
			Completing the squares yields
			\begin{equation} \label{eq:14}
			\begin{split}
			\|  \tilde\xv_{t+1} \|_{\SMM_{t+1}^{-1}}^2 \leq 
			& (1 + \frac{1}{\alpha}) \| \tilde \xv_{t} \|_{\SMM_t^{-1}}^2 \\
			& + (1 + \frac{1}{\alpha}) \| \yv_t - \CM \bar \xv_t \|^2_{\VM^{-1}}\\
			& - (1 + \frac{1}{\alpha})  
			\left(\yv_t - \CM \hat \xv_t\right)^\top \VM^{-1} \\
			& . \left(\CM \SMM_t \CM^\top \VM^{-1} + \IM_p \right)^{-1} \left(\yv_t - \CM \hat \xv_t\right) \\
			& + (1 + \alpha) \|  \bar \xv_{t+1} - \AM \bar \xv_t \|^2_{\SMM_{t+1}^{-1}}.    \\
			\end{split}
			\end{equation}
			Next, summing over $t = 0,2, \dots, T - 1$, and using the telescoping property we get that
			\begin{equation}	
			\begin{split}
			L_T & \leq 
			\bar r V_T + \bar r \| \bar \xv_0 \|^2 \\
			& + \bar r \left( \frac{1}{\alpha} \sum_{t = 0}^{T-1}\| \tilde \xv_t \|_{\SMM_t^{-1}}^2  
			+ \alpha \sum_{t = 0}^{T-1} \|  \bar \xv_{t+1} - \AM \bar \xv_t \|_{\SMM_{t+1}^{-1}}^2 \right) ,
			\end{split}
			\end{equation}
			from which the claim follows.  \end{proof}
		
		Define
		\begin{equation} \label{WT}
		\begin{cases}
		W_T = \sum_{t=0}^{T - 1} \|  \bar \xv_{t+1} - \AM \bar \xv_t \|^2 \\
		\tilde X_T = \sum_{t=0}^{T - 1} \| \tilde \xv_t \|^2,
		\end{cases}
		\end{equation}
		where $W_T$ denotes the total drift 
		(cumulative squared norm of process residuals) 
		and $\tilde X_T$ indicates the cumulative state estimation error. 
		We use next Lemma to upper bound explicitly $\tilde X_T$.
		
		\begin{Lemma} 
			Consider the linear system 
			\begin{equation} 
			\tilde \xv_{t+1} = \HM \tilde \xv_t + \bar \wv_t ,
			\end{equation}
			with $\bar \sigma (\HM) < 1$. Then
			\begin{equation} \label{Xtildebound}
			\begin{split}
			\tilde X_T  \leq 2 b \| \tilde \xv_0 \|^2 
			+ 4 c \sum_{t = 0}^{T - 1} \| \bar \wv_{ t} \|^2  .
			\end{split}
			\end{equation}
			where 
			\begin{equation} \label{c}
			b = 1/\left(1 - \bar \sigma^2 (\HM) \right), \quad c = \left(1 + \bar \sigma^2 (\HM) \right) / \left( 1 - \bar \sigma^2 (\HM) \right)^3.
			\end{equation}
		\end{Lemma}
		\begin{proof}
			For any $t = 0, \dots T - 1$, one has
			\begin{equation} \label{Hinf}
			\tilde \xv_t = \HM^t \tilde \xv_0 + \sum_{k = 0}^{t - 1} \HM^{k} \bar \wv_{t-k-1}.
			\end{equation}
			Using the triangular inequality, we get that
			\begin{equation}
			\| \tilde \xv_t \| \leq  \bar \sigma^{t} (\HM) \| \tilde \xv_0 \|
			+ \sum_{k = 0}^{t - 1} \bar \sigma^{k} (\HM) \| \bar \wv_{ t-k-1} \| .
			\end{equation}
			Hence
			\begin{equation}
			\begin{split}
			\| \tilde \xv_t \|^2 \leq 2 \bar \sigma^{2t} (\HM) \| \tilde \xv_0 \|^2 
			+ 2 \left(\sum_{k = 0}^{t - 1} \bar \sigma^{k} (\HM) \| \bar \wv_{ t-k-1} \|\right)^2 ,
			\end{split}
			\end{equation}
			which can be written as
			\begin{equation} \label{eqeq}
			\begin{split}
			\| \tilde \xv_t \|^2 & \leq 2 \bar \sigma^{2t} (\HM) \| \tilde \xv_0 \|^2 \\
			& + 2 \left( \frac{\sum_{k = 0}^{t - 1} \frac{1}{(k + 1)^2} \left(\Gamma_t (k + 1)^2 \bar \sigma^{k} (\HM) \| \bar \wv_{ t-k-1} \|\right)} {\Gamma_t}  \right)^2 ,
			\end{split}
			\end{equation}
			where $\Gamma_t = \sum_{k = 0}^{t - 1} \frac{1}{(k + 1)^2}$. Next, we use Jensen's inequality to bound the second term in (\ref{eqeq}) as
			\begin{equation}
			\begin{split}
			\| \tilde \xv_t \|^2 \leq 2 \bar \sigma^{2t} (\HM) \| \tilde \xv_0 \|^2 
			+ 2  \sum_{k = 0}^{t - 1} \Gamma_t (k + 1)^2 \bar \sigma^{2k} (\HM) \| \bar \wv_{ t-k-1} \|^2  ,
			\end{split}
			\end{equation}
			Summing over $t = 0, \dots, T-1$, yields
			\begin{equation} 
			\begin{split}
			\tilde X_T \equiv \sum_{t=0}^{T-1} \| \tilde \xv_t \|^2 & \leq 2 \sum_{t = 0}^{T - 1} \bar \sigma^{2t} (\HM) \| \tilde \xv_0 \|^2 \\
			& + 2 \sum_{t = 0}^{T - 1} \Gamma_t \sum_{k = 0}^{t - 1} (k + 1)^2 \bar \sigma^{2k} (\HM)\| \bar \wv_{ t-k-1} \|^2 ,
			\end{split}
			\end{equation}
			which can be written as
			\begin{equation}
			\begin{split}
			\tilde X_T & \leq 2 \sum_{t = 1}^{T} \bar \sigma^{2(t-1)} (\HM) \| \tilde \xv_0 \|^2 \\
			& + 2 \sum_{t = 1}^{T} \Gamma_{t-1} \sum_{k = 0}^{t - 1} (k + 1)^2 \bar \sigma^{2k} (\HM)\| \bar \wv_{ t-k} \|^2 .
			\end{split}
			\end{equation}
			Using the geometric series convergence in the first term, and change of indexing in the second term, we get
			\begin{equation} \label{eq:Xtilde_T}
			\begin{split}
			\tilde X_T  \leq 2 b \| \tilde \xv_0 \|^2 
			+ 2 \sum_{t = 1}^{T} \Gamma_{t-1} \| \bar \wv_{ t} \|^2 \sum_{k = 0}^{t - 1} (k + 1)^2 \bar \sigma^{2k} (\HM) .
			\end{split}
			\end{equation}
			The Riemann zeta function and the polylogarithmic series are upper bounded as
			\begin{equation} \label{zeta}
			\Gamma_{t-1} = \sum_{k = 1}^{t} \frac{1}{k^2} \leq 2 \qquad \forall t \geq 1,
			\end{equation} 
			and
			\begin{equation} \label{polylog}
			\sum_{k = 0}^{t - 1} (k + 1)^2 (\bar \sigma^2 (\HM))^{k} \leq \frac{\left(1 + \bar \sigma^2 (\HM) \right)}{\left( 1 - \bar \sigma^2 (\HM) \right)^3},
			\end{equation}
			respectively. Substituting (\ref{zeta}) and (\ref{polylog}) in (\ref{eq:Xtilde_T}) gives the result. 
		\end{proof}
		Finally, using the above lemmas, we prove the main result of this section.
		\begin{Theorem}
			Suppose $\left(\CM, \AM\right)$ is detectable and $\left(\AM, \QM\right)$ is controllable. Let the KF algorithm \ref{algkf} be run on the sequence $ S = \{\yv_0,  \yv_1 , \dots, \yv_{T-1} \}$ with initial values be chosen as $\hat \xv_0 = \mathbf{0}$ and $\SMM_0 = \IM_n$, generating the estimates $\{\hat \xv_0, \hat \xv_1, \dots, \hat \xv_{T} \}$
			. Then for any sequence of targets $\{\bar \xv_0, \bar \xv_1, \dots, \bar \xv_T  \}$, one has 
			\begin{equation} \label{B1}
			\begin{split}
			L_T \leq  & \bar r V_T  + \bar r \| \bar \xv_0 \|^2 \\
			& + 2 \bar r a \sqrt{ 2 W_T\left( b \| \bar \xv_0 \|^2 + 4 c \left( W_T + \bar \sigma(\bar \KM^\top \bar \KM) V_T \right) \right) },
			\end{split}
			\end{equation}
		\end{Theorem}
		with $\bar \KM$ is the steady-state Kalman gain being given in (\ref{Kbar}), and where constants $\bar r$, $a$, $b$, and $c$ are defined as in (\ref{rbar}) and (\ref{c}), respectively.
		
		\begin{proof}
			From Lemma 2 we have $\| \bar \wv_t \|^2 =\| \wv_t - \bar \KM \vv_t \|^2$. Hence
			\begin{equation} \label{wbar}
			\sum_{t = 0}^{T-1} \| \bar \wv_t \|^2 \leq 2 \left(W_T + \bar \sigma \left(\bar \KM^\top \bar \KM \right) V_T\right).
			\end{equation}
			Using (\ref{wbar}) in (\ref{Xtildebound}) one obtains
			\begin{equation} \label{xtildebound}
			\begin{split}
			\tilde X_T  \leq 2 b \| \bar \xv_0 \|^2 
			+ 8 c \left(W_T + \bar \sigma \left(\bar \KM^\top \bar \KM \right) V_T\right).
			\end{split}
			\end{equation}
			Now, substituting (\ref{xtildebound}) in (\ref{reg}) and setting $\alpha$ as (see Proposition 1 in Appendix)
			\begin{equation}
			\alpha = \sqrt{\frac{2 \left( b \| \bar \xv_0 \|^2 + 4 c \left( W_T + \bar \sigma(\bar \KM^\top \bar \KM) V_T \right) \right)}{W_T}}, 
			\end{equation} 
			gives the result.
		\end{proof}
		\begin{Remark}
			One main conclusion is that the constant factor $\bar r$ in (\ref{B1}) is strictly greater than 1 which implies that the absolute regret 
			\begin{equation}
			\begin{split}
			\mathcal{R_T} = L_T - V_T  & \leq \bar \sigma \left(\CM \Sigma \CM^\top \VM^{-1} \right) V_T \\
			& + \mathcal{O} \left( \sqrt{W_T \left(W_T + V_T \right)} \right),
			\end{split}
			\end{equation} 
			depends linearly on 
			$V_T$. Note that as is implied by the $H_{\infty}$ norm lower bound obtained in \cite{Hassibi-kf}, the case $c_1 = 1$ (see Remark \ref{Rem2}) is not achievable in the worst-case bound of the KF.  { If the cumulative loss $V_T$ is small and if the reference predictor sequence evolves approximately as $\bar x_{t+1} = A \bar x_t$, then $W_T = o(T)$ would be small, leading to a vanishing regret bound. On the other hand, if the cumulative loss $V_T$ is large (linear in $T$) and the offline reference predictor is not evolving as $\bar x_{t+1} = A \bar x_t$, then $W_T$ would be a large quantity and the bound would be $\mathcal{O}(T)$. }
			
		\end{Remark}
		
		{
			\begin{Remark}
				A number of known/extreme cases can be considered. For instance, suppose that the process noise $\wv_t$ equals the state $\bar \xv_t$ itself. Then, from (\ref{wt}) we get
				\begin{equation}
				\bar x_{t+1} =  \underbrace{(\AM + \IM )}_{\bar \AM}  \bar x_t,
				\end{equation}
				If $\bar \AM$ is stable, this yields 
				$\bar \xv_t \rightarrow 0$ as $t \rightarrow \infty$.
				Thus the total drift $W_T$ converges to zero and the bound becomes equivalent to the worst-case bound with stationary target. On the other hand, if $\bar \AM$ is unstable, $\bar \xv_t$ would grow unbounded, and the total drift $W_T$ will be quite large and one would get a rather large bound indicating a poor performance of the KF. 
				
			\end{Remark}
		}
		\begin{Remark}
			For the {\em special} case of the RLS algorithm where 
			$$ \AM = \IM_n \quad \QM = \mathbf{0} \quad \CM = \hv_t^\top \quad \VM = \IM_p ,$$
			the following worst-case bound obtained by Vaits et al. \cite{Vaits-15}
			\begin{equation}
			L_T \leq L_T {\left( \{\bar \xv_t \}\right)} + \mathcal{O} \left(T^{\frac{2}{3}}\left( W_T(\{\bar \xv_t \}) \right)^{\frac{1}{3}} \right),
			\end{equation}
			with $L_T {\left( \{\bar \xv_t \}\right)}$ indicating the cumulative loss of any sequence of targets $\{\bar \xv_t \}$, and where $W_T (\{\bar \xv_t \}) = \sum_{t = 0}^{T - 1} \| \bar \xv_{t+1} - \bar \xv_{t}\|^2$ denotes the total drift (See \cite{Vaits-15}, Section VI for details). One can observe that the factor 1 multiplying the term $L_T {\left( \{\bar \xv_t \}\right)}$, leading to a vanishing tracking regret bound when the total drift is sublinear. On the other hand, the RLS dependency to the total drift $\mathcal{O} \left(T^{\frac{2}{3}}\left( W_T(\{\bar \xv_t \}) \right)^{\frac{1}{3}} \right)$ is worse than the KF bound dependency $\mathcal{O} \left( \sqrt{W_T \left(W_T + V_T \right)} \right)$ when the cumulative drift is sublinear in $T$. 
		\end{Remark}

		\section{Comparison with the $H_{\infty}$ Setting}

		This section provides a comparison with the $H_{\infty}$ setting used in adaptive filtering. 
		Motivated by the assumption that the data 
		is close to the linear time-invariant state-space model 
		(\ref{sysdyn}), analyses of $H_\infty$ estimation aims to bound the maximum energy gain from the unknown disturbances to the state estimation error. 
		
		\begin{Theorem} \cite{Hassibi-kf}
			({\em $H_{\infty}$ norm bound of the Kalman filter}) Consider the system dynamic (\ref{sysdyn}). Run the KF algorithm 1 and let the initial values be chosen as $\hat \xv_0 = \mathbf{0}$ and $\SMM_0 = \IM$. 
			Then for any $T$, one has
			\begin{equation} \label{Hinfnormbound}
			\begin{split}
			& \sup_{\bar \xv_0, \wv , \vv \in l_2}
			\frac{\sum_{t = 0}^{T - 1}\| \tilde \xv_t \|_{\CM^\top \VM^{-1}\CM}^2}{  \| \bar \xv_0 \|^2 + \sum_{t=0}^{T - 1}\| \vv_t \|_{\VM^{-1}}^2 + \sum_{t=0}^{T - 1} \| \wv_t \|^2_{\QM^{-1}} } \\
			& \leq \left( \sqrt{\bar r} + 1 \right)^2,
			\end{split}
			\end{equation} 
		\end{Theorem} 
		where $\bar r$ is defined as in (\ref{rbar}), and $l_2$ is the space of square-summable sequences.
		
		For a fair comparison, we derive a worst-case loss bound from the 
		$H_{\infty}$ norm bound (\ref{Hinfnormbound}) of the KF.
		
		\begin{Corollary}
			Consider the steady-state KF and let the initial values be chosen as
			$\hat \xv_0 = \mathbf{0}$ and $\SMM_0 = \IM_n$. Then for all sequence of observations $\{\yv_0, \yv_1, \dots \}$
			and for any sequence of targets $\{ \bar \xv_0, \bar \xv_1, \dots  \}$, 
			the total loss suffered by the algorithm is bounded by
			\begin{equation} \label{B3}
			\begin{split}
			L_T \leq & \left(1 + \left( \sqrt{\bar r} + 1\right)^2 \right) V_T + \left( \sqrt{\bar r} + 1\right)^2  \| \bar \xv_0 \|^2 \\
			& + \left( \sqrt{\bar r} + 1\right)^2 \bar \sigma (\QM^{-1}) W_T  \\
			&  + 2 \left( \sqrt{\bar r} + 1\right) \sqrt{V_T ( \| \bar \xv_0 \|^2 + V_T + \bar \sigma (\QM^{-1}) W_T)} ,
			\end{split}
			\end{equation} 
		\end{Corollary}
		where $W_T$ is defined as in (\ref{WT}).
		
		\begin{proof}
			We re-write the $H_{\infty}$ norm bound (\ref{Hinfnormbound}) as
			\begin{equation} \label{Hinf1}
			\sum_{t = 0}^{T - 1} \| \tilde \xv_t \|_{\CM^\top \VM^{-1} \CM}^2 \leq \left( \sqrt{\bar r} + 1 \right)^2 \left( \| \bar \xv_0\|^2 + V_T + \bar \sigma (\QM^{-1}) W_T \right) .
			\end{equation}
			Using Lemma 5 in Appendix for any $\alpha > 0$, one has that
			\begin{equation} \label{eq:40}
			\begin{split}
			\| \yv_t - \CM \hat \xv_t \|_{\VM^{-1}}^2 \leq (1 & + \frac{1}{\alpha}) \| \yv_t - \CM \bar\xv_t \|_{\VM^{-1}}^2 \\
			&  + (1 + \alpha) \| \tilde \xv_t \|^2_{\CM^\top \VM^{-1} \CM} .
			\end{split}
			\end{equation} 
			Summing over $ t = 0, \dots, T - 1$, and then substituting (\ref{Hinf1}) in (\ref{eq:40}), one gets
			\begin{equation} \label{LT}	
			\begin{split}
			L_T & \leq (1 + \frac{1}{\alpha}) V_T \\
			& + (1 + \alpha) \left( \sqrt{\bar r} + 1\right)^2 
			\left( \| \bar \xv_0 \|^2 + V_T + \bar \sigma (\QM^{-1}) W_T \right) , 
			\end{split}
			\end{equation}
			Now, setting $\alpha$ as
			\begin{equation} \label{alfa}
			\alpha = \sqrt{\frac{V_T}{\left( \sqrt{\bar r} + 1\right)^2 ( \| \bar \xv_0 \|^2 + V_T + \bar \sigma (\QM^{-1}) W_T)}},
			\end{equation}
			gives the result. 
		\end{proof}

		\begin{Remark}
			Comparing the KF bounds in (\ref{B1}) and (\ref{B3}), one observes that the multiplicative factors 
			$\left(1 + \left(\sqrt{\bar r} + 1\right)^2 \right)$ and $\left( \sqrt{\bar r} + 1 \right)^2$ in the first two terms are worse for the bound (\ref{B3}) than for the bound (\ref{B1}) which equals $\bar r$. The KF bound (\ref{B1}) has dependency $\mathcal{O} \left( \sqrt{ W_T \left(W_T + V_T\right) } \right)$ to the cumulative drift while the dependency of the bound (\ref{B3}) is
			$\mathcal{O} \left( W_T + \sqrt{  V_T \left( W_T + V_T \right) } \right)$. However, the constant multiplicative factor to the drift, is worse for the bound in (\ref{B1}) than the bound in (\ref{B3}). 
			On the other hand, the bound (\ref{B3}) is worse if the cumulative loss $V_T$  is linear in $T$ which implies that first and last terms of the bound (\ref{B3}) become dominant and hence the multiplicative factors make the bound (\ref{B3}) worse than the bound (\ref{B1}).
			It is worth mentioning that both bounds derived from a \emph{minimax} approach and small bounds are obtained 
			in either case if $W_T = o(T)$ and the cumulative loss $V_T$ is small .	
		\end{Remark}

		\section{Simulation}
		
		We finish the paper with two sets of experiments illustrating 
		the use of the worst-case bounds of the KF algorithm. For both experiments the 
		discrete system dynamic matrices 
		$\AM \in \R^{n \times n}$ and $\CM \in \R^{p \times n}$ are generated randomly 
		using the {\em 'drss'} command in \emph{Matlab}. We set $T = 2000$, $n = 10$, $p = 4$.
		The weighting matrices are chosen as $\QM = 0.5 \IM$ and $\VM = \IM$. 
		The system initial states are set to zero. 
		For choosing $\wv_t$, two different cases
		of linear and sublinear drift are considered.
		In the first experiment a sequence of vector $\wv_t \in \R^{10}$ 
		is generated from a Gaussian distribution for which the instantaneous 
		drift $\| \bar \xv_{t+1} - \AM \bar \xv_{t} \|^2$ is constant (linear drift), 
		obeying $\yv_t = \CM \bar \xv_t + \vv_t$, where $\vv_t \sim \mathcal{N}(\mathbf{0}, \IM)$.
		The second experiment is based on sublinear drift. 
		Here, we use a polynomial decay of the drift, $\| \bar \xv_{t+1} - \AM \bar \xv_{t} \|^2 \leq t^{-\beta}$ for some $\beta > 0$. In this case $W_T \leq \log(T) + 1$ for $\beta = 1$, and $W_T \leq \frac{T^{(1 - \beta)} - \beta}{1 - \beta}$ otherwise. We set $\beta = 0.5$ so that the sequence of vectors $\wv_t \in \R^{10}$ rotating along a unit circle in a rate of $t^{-0.5}$ (sublinear drift). 
		
		Fig. 1 displays how the behavior of the worst-case 
		loss bounds for the KF algorithm 1 in (\ref{B1}) (KF bound 1)
		and the worst-case loss bound (\ref{B3}) 
		obtained form the $H_{\infty}$ norm bound (\ref{Hinfnormbound}) (KF bound 2), when 
		$t = 1, 2, \dots, 2000$ and the total drift is linear (upper panel) and sublinear (lower panel). 
		One can observe that the performance of the algorithm
		depends on the total amount of drift for which the algorithm performs worse in severe conditions (linear drift). 
		The worst-case KF-bound
		1 has superior performance than KF bound 2 when the total drift is sublinear whereas it has worse performance in case where the total drift is linear. It should be stressed that the worst-case
		bounds depend on the specific choice of the design parameters and system dynamic matrices. 
		
		The development of the empirical cumulative loss function of the KF algorithm, and empirical cumulative loss function of the sequence of targets is demonstrated 
		in Fig. 2. It is observed that the performance of the KF algorithm converges to the performance of the best sequence of targets when $W_T = o(T)$.  
		
		\begin{figure} \label{fig1} 
			\centering
			\includegraphics[width = 6in]{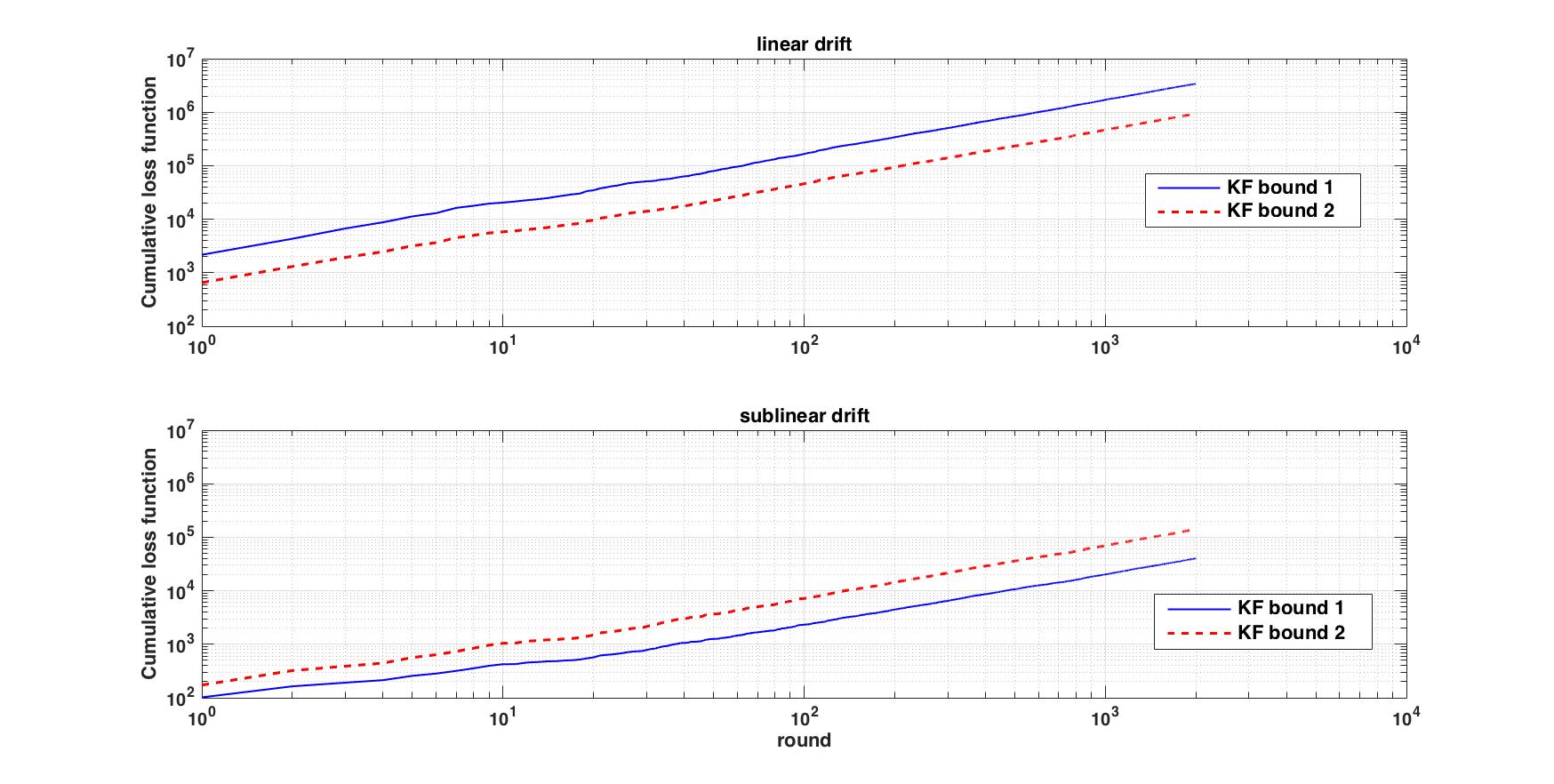}
			\caption{Comparison of the behavior of worst-case loss bounds in (\ref{B1}) and (\ref{B3}) for two experiments where the total drift is (upper panel) linear and (lower panel) sublinear.}
		\end{figure}
		
		\begin{figure} \label{fig2} 
			\centering
			\includegraphics[width = 6in]{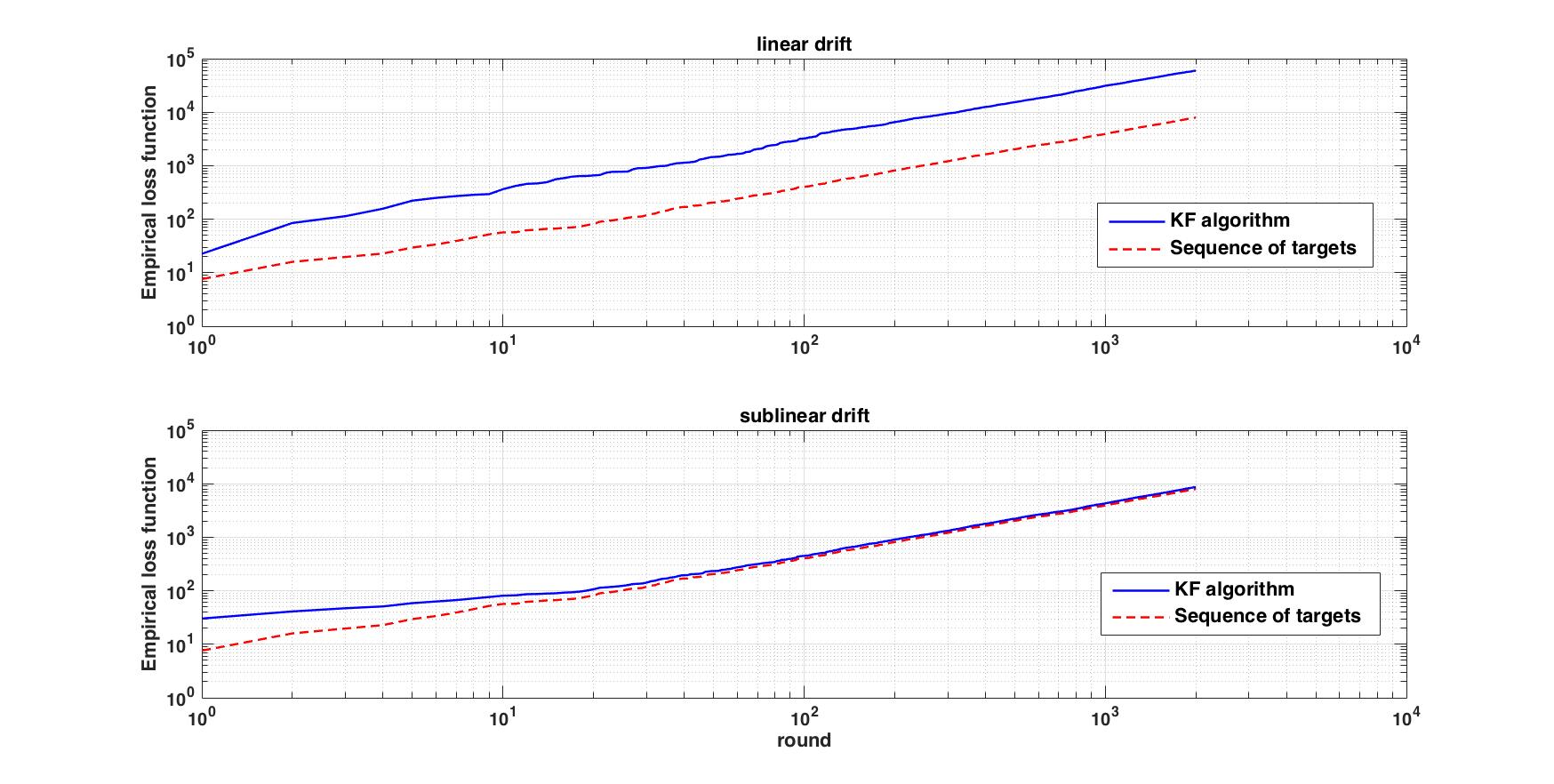}
			\caption{Empirical cumulative loss of the KF algorithm and that of the sequence of targets as a function of time when the drift is (upper panel) linear and (lower panel) sublinear.}
		\end{figure}
		
		\section{Conclusion and Future Works}
		This paper studied the worst-case performance of the KF algorithm using novel universal prediction perspective in online machine learning. 
		Tracking worst-case bounds 
		were proved and compared with the 
		bound derived from an $H_{\infty}$ setting. 
		The proposed bounds hold for a wide range of observation models and noise distributions. 
		The results do not require the data to be linearly related
		as opposed to $H_{\infty}$ setting which is based on 
		the assumption that data obeys a fixed linear dynamical model.
		It was shown that 
		the worst-case bounds scales proportional with the deviation of the comparator from a known predetermined dynamical model. 
		
		Future works will extend 
		the results to more general problem where the target 
		is any arbitrary linear combination of the state. 
		{An interesting open problem is to incorporate a time-varying dynamical model 
			in which case the tracking bounds will scale with the deviation of a reference predictor from the best sequence of dynamical models.
			It would also be interesting to extend the results
			to the case where the dimensionality of the state vector $\bar \xv_t \in \R^n$ is very high (e.g., $n \gg T$).
			Straightforward application of previous results will
			render the bound uninformative since the bound scales up linear 
			in the state dimension.
			Some recent work has examined the role of high-dimensionality in online RLS algorithms \cite{Angelosante10} and online mirror descent methods \cite{Hall15}, under the assumption that the problem has a proper sparsity structure.  
			A more general approach is to employ
			random projection technique and the celebrated Johnson-Lindenstrauss lemma which states roughly that, given an arbitrary set of $n$ points in a high-dimentional Euclidean space, there exists a linear map of this points in a low-dimensional Euclidean space such that all pairwise distances are preserved \cite{Ailon}, \cite{Vempala}. The goal would be to analyse the performance of the random projected KF algorithm relative to the performance of the computationaly intractable offline high-dimensional case.  }
		
		\appendix
		
		The proofs of the worst-case bounds make use of the following auxiliary results.
		\begin{Lemma} [Hassibi and Kailath \cite{Hassibi-kf}] \label{lem2a}
			For any vectors $\av$, $\bv$, any matrix $\MM \succ 0$, and for all $\alpha > 0$ we have
			\begin{equation}
			(\av + \bv)^\top \MM (\av + \bv) \leq \left(1+ \frac{1}{\alpha}\right) \av^\top \MM \av + (1+ \alpha) \bv^\top \MM \bv .
			\end{equation}
		\end{Lemma}
		
		\begin{Lemma} \label{lem1a}
			Assuming that $\AM$ is nonsingular, the following statement holds,
			\[
			\SMM_{t+1}^{-1} \preceq \AM^{-\top} \left(\SMM_t^{-1} + \CM^\top \VM^{-1} \CM \right) \AM^{-1},
			\]
		\end{Lemma}
		where $\SMM_t$ is the solution of the Riccati recursion  (\ref{riccati}).
		
		\begin{proof}
			From the Riccati recursion (\ref{riccati}) and using Woodbury matrix inversion lemma,
			$\SMM_{t+1}^{-1}$ is obtained as
			\begin{equation} \label{eq:11}
			\SMM_{t+1}^{-1} = \bar{\mathbf{\Pi}}^{-1}_t - \QM^{\frac{1}{2}} \bar{\mathbf{\Pi}}^{-1}_t \left(\IM + \QM^{\frac{1}{2}} \bar{\mathbf{\Pi}}^{-1}_t \QM^{\frac{1}{2}} \right)^{-1} \bar{\mathbf{\Pi}}^{-1}_t \QM^{\frac{1}{2}} ,
			\end{equation} 
			where $\bar{\mathbf{\Pi}}_t = \AM \left( \SMM_t^{-1} + \CM^\top \VM^{-1} \CM \right)^{-1} \AM^\top$. 
			Since the second term 
			in (\ref{eq:11}) is positive semidefinite, one gets
			\begin{equation}
			\SMM_{t+1}^{-1} \preceq \bar{\mathbf{\Pi}}^{-1}_t = \AM^{-\top} \left(\SMM_t^{-1} + \CM^\top \VM^{-1} \CM \right) \AM^{-1} .
			\end{equation}
			This completes the proof. 
		\end{proof}
		
		\begin{Proposition} \label{propvanilla}
			Let $a, b > 0$ be constants, then
			\begin{equation}
			\underset{\xi > 0}{\mathrm{inf}}  \frac{a}{\xi} + \xi b = 2 \sqrt {ab} .
			\end{equation}
		\end{Proposition}
		\begin{proof}
			This is seen by choosing $\xi = \sqrt \frac{a}{b}$, obtained by equating the derivative to zero.
		\end{proof}
		
		
		\ifCLASSOPTIONcaptionsoff
		\newpage
		\fi

\end{document}